\documentclass[12pt]{amsart}

\usepackage[colorlinks=true]{hyperref}
\usepackage{amssymb}
\usepackage[all]{xy}
\usepackage{mathtools}
\usepackage{stackrel}
\usepackage{stmaryrd}

\usepackage{todonotes}
\usepackage{eucal}

\usepackage{tikz}
\usetikzlibrary{calc,matrix}
\tikzset{w/.style={circle, draw,inner sep=1pt},b/.style={circle,draw,fill,inner sep=2pt}, s/.style={rectangle, draw,inner sep=3pt}}

\newcommand{\nocontentsline}[3]{}
\newcommand{\tocless}[2]{\bgroup\let\addcontentsline=\nocontentsline#1{#2}\egroup}

\newcommand{\bL}{\mathbb{L}}

\newcommand{\bT}{\mathbb{T}}

\newcommand{\Z}{\mathbb{Z}}

\newcommand{\cC}{\mathcal{C}}

\newcommand{\cO}{\mathcal{O}}

\newcommand{\e}{\epsilon}

\newcommand{\Arr}{\mathrm{Arr}}

\newcommand{\Der}{\mathrm{Der}}
\newcommand{\dAff}{\mathrm{dAff}}
\newcommand{\dgmix}{\mathrm{dg}^{gr,\epsilon}}
\newcommand{\bdgmix}{\mathbf{dg}^{gr,\epsilon}}
\newcommand{\DR}{\mathrm{DR}}

\newcommand{\End}{\mathrm{End}}

\newcommand{\Hom}{\mathrm{Hom}}
\newcommand{\id}{\mathrm{id}}

\newcommand{\Map}{\mathrm{Map}}
\newcommand{\bMap}{\boldsymbol{\mathrm{Map}}}

\newcommand{\sSet}{\mathrm{sSet}}

\newcommand{\Sym}{\mathrm{Sym}}
\newcommand{\Symp}{\mathbf{Symp}}

\newcommand{\dg}{\mathrm{dg}}
\newcommand{\CAlg}{\mathrm{CAlg}}
\newcommand{\bCAlg}{\mathbf{CAlg}}

\newcommand{\bZ}{\mathbf{Z}}

\newcommand{\cdga}{\mathrm{cdga}}
\newcommand{\triv}{\mathrm{triv}}

\newcommand{\QCoh}{\mathrm{QCoh}}

\newcommand{\Perf}{\mathrm{Perf}}
\newcommand{\bPerf}{\boldsymbol{\mathrm{Perf}}}
\newcommand{\dSt}{\mathrm{dSt}}

\newcommand{\cF}{\mathcal{cF}}

\newcommand{\bCAlgmix}{\mathbf{CAlg}^{gr,\epsilon}}
\newcommand{\bCoAlgd}{\mathbf{CoAlgd}}

\newcommand{\PrSymp}{\mathbf{PrSymp}}
\newcommand{\QF}{\mathbf{QF}}

\DeclareMathOperator{\Spec}{Spec}

\newtheorem{thm}{Theorem}[section]

\newtheorem{prop}[thm]{Proposition}

\theoremstyle{definition}
\newtheorem{defn}[thm]{Definition}

\theoremstyle{remark}
\newtheorem{remark}[thm]{Remark}

\begin{document}
\address{Department of Mathematics,
The University of British Columbia, 1984 Mathematics Road,
Vancouver, BC -
Canada V6T 1Z2}
\email{samuelbach@math.ubc.ca}
\email{samuelbach@hotmail.fr}
\author{Samuel Bach}
\title{The derived moduli stack of shifted symplectic structures}
\address{Dipartimento di Matematica, Universit\`a di Milano, Via Cesare Saldini 50, 20133 Milan, Italy}
\email{valerio.melani@unimi.it}
\email{valerio.melani@outlook.com}
\author{Valerio Melani}
\begin{abstract}
We introduce and study the derived moduli stack $\Symp(X,n)$ of $n$-shifted symplectic structures on a given derived stack $X$, as introduced in \cite{PTVV}. In particular, under reasonable assumptions on $X$, we prove that $\Symp(X,n)$ carries a canonical quadratic form, in the sense of \cite{Ve}. This generalizes a classical result of \cite{FH}, which was established in the classical $C^{\infty}$-setting, to the broader context of derived algebraic geometry, thus proving a conjecture stated in \cite{Ve}.
\end{abstract}
\maketitle

\tableofcontents

\addtocontents{toc}{\protect\setcounter{tocdepth}{2}}

\section*{Introduction}

Let $M$ be a closed smooth manifold. One natural question in symplectic geometry is to classify all possible symplectic structures on $M$: a reasonable approach to this is to study the moduli space $\Symp(M)$ of symplectic structures on $M$. The space $\Symp(M)$ can be studied from the point of view of symplectic topology (see for example \cite{MT}, \cite{Sm}, \cite{Vi} to name a few), but in this paper we will rather be interested in its geometry. One of the main results in this direction is given by Fricke and Habermann, who in \cite{FH} construct a (pseudo-)Riemaniann structure on $\Symp(M)$. 

The purpose of this paper is to extend the results of \cite{FH} to the setting of derived algebraic geometry. Derived algebraic geometry can be informally understood as the study of generalized spaces (i.e. derived stacks), whose local models are derived commutative algebras, that is to say simplicial commutative algebras. If we
suppose to be working over a base field $k$ of characteristic zero, the local models can
also be taken to be non-positively graded commutative dg algebras. We refer for example to \cite{To} (and references therein) for a more precise survey.

In the seminal paper \cite{PTVV}, the authors introduced the notion of $n$-shifted symplectic structure on a given derived stack $X$, where $n$ is any integer. On the other hand, there is a parallel theory of shifted quadratic forms on derived stacks, developed in \cite{Ve} and \cite{Ba}.
Building on these works, we construct a derived moduli stack $\Symp(X,n)$ of $n$-shifted symplectic structures on $X$, which has to be thought as a derived enhancement (in the algebraic setting) of the moduli space of symplectic structures studied in \cite{FH}. 

The main result of the present work can be stated as follows.

\begin{thm}
Let $X$ be a nice enough derived stack, and let $n$ be an integer. Then the derived moduli stack $\Symp(X,n)$ of $n$-shifted symplectic structures on $X$ carries a natural quadratic form, in the sense of \cite{Ve} and \cite{Ba}, extending the one of \cite{FH}.
\end{thm}

In particular, the above Theorem thus proves a conjecture which was stated in \cite[Remark 3.15]{Ve}.

The paper is organized as follows. In Section 1, we recall some preliminary notions that will be used later on. In Section 2, we give a construction of the derived moduli stack of Lie coalgebroids $\bCoAlgd(X)$ on a derived stack $X$. Moreover, passing through the important notion of symplectic Lie coalgebroid, we arrive to the definition of the derived moduli stack of $n$-shifted pre-symplectic structures $\PrSymp(X,n)$. Section 3 is devoted to the computation of the cotangent complex of $\Symp(X,n)$, which is the sub-stack of $\PrSymp(X,n)$ whose points correspond to $n$-shifted symplectic structures on $X$. The techniques used here are similar to the ones in \cite{Ba}.
Finally, in the last section we construct a shifted quadratic form on $\Symp(X,n)$, proving our main result. 

Let us mention that in \cite{FHH}, the authors gave a criterion for the non-degeneracy of the pseudo-Riemaniann structure of \cite{FH}. It would be interesting to see if one can generalize their arguments to derived algebraic geometry. In other words, one should be able to characterize the non-degeneracy of the shifted quadratic structure we construct in the present paper. We plan to come back to this question in a future work.
\subsection*{Acknowledgements}

We would like to thank M. Porta and G. Vezzosi for interesting discussions related to the subject of this paper. Moreover, we thank B. To\"en for pointing out a missing argument in a previous version, as well as the anonymous referee for his/her useful comments. The work of V.M. was partially supported by FIRB2012 ``Moduli spaces and applications''.

\subsection*{Notations and conventions}

\begin{itemize}

\item The $\infty$-category of simplicial sets is denoted by $\sSet$.

\item Let $\cC$ be an $\infty$-category. Given any two objects $X,Y$ of $\cC$, we will denote by $\Map_{\cC}(X,Y) \in \sSet$ the mapping space between them. If moreover $\cC$ is a closed monoidal $\infty$-category, we will write $\bMap_{\cC}(X,Y) \in \cC$ for the internal hom-object of $\cC$.

\item $k$ denotes a base field of characteristic 0. We denote by $dg_k$ the ordinary category of cochain complexes of $k$-vector spaces. This has the usual projective model structure, where weak equivalences are given by quasi-isomorphisms, and fibrations by surjections. We write $\dg_k$ for the associated $\infty$-category. The standard tensor product of complexes makes $\dg_k$ a symmetric monoidal $\infty$-category. More generally, if $A$ is a commutative dg algebra over $k$, then $\dg_A$ will denote the symmetric monoidal $\infty$-category of dg $A$-modules.
Given two $A$-modules $M$ and $N$, we denote by $\Map_A(M,N) \in \sSet$ the mapping space in $\dg_A$.
\item The ordinary category of commutative dg algebras concentrated in non-positive degree will be denoted $cdga^{\leq 0}$. It has the usual model structure for which weak equivalences are quasi-isomorphisms, and fibrations are surjections in negative degrees. We write $\cdga^{\leq 0}$ for the associated $\infty$-category. The $\infty$-category $\dAff$ of derived affines is simply the opposite category $(\cdga^{\leq 0})^{op}$.

\item The $\infty$-category of \emph{derived stacks} (in the sense of \cite[Section 2.2]{TVe}) with values in $\sSet$ over $k$ is denoted $\dSt$. The direct product makes $\dSt$ a closed monoidal $\infty$-category.

\item Given $A \in \dAff^{op}$, we also use the notations $\QCoh(A)$ for the symmetric monoidal $\infty$-category of dg $A$-modules $\dg_A$, while $\Perf(A)$ will denote the full subcategory of \emph{perfect} dg $A$-modules. If $X$ is a derived stack, we define the symmetric monoidal $\infty$-categories
\[ \QCoh(X): = \lim_{\Spec A \to X} \QCoh(A)\ \ \ \ \ \ \text{and}\ \ \  \ \ \ \Perf(X) :=\lim_{\Spec A \to X} \Perf(A), \]
where the limits are taken in the $\infty$-category of stable symmetric monoidal presentable $\infty$-categories.
If $\cF \in \Perf(X)$ is a perfect complex, we will denote by $\cF^{\vee} \in \Perf(X)$ its $\cO_X$-linear dual.

\item If $\cC$ is an $\infty$-category, we denote by $\cC^{\sim}$ the maximal $\infty$-subgroupoid contained in $\cC$. We use the notation $\Arr(\cC)$ for the $\infty$-category of morphisms in $\cC$.

\item If $\cC$ is an $\infty$-category, we denote by $\cC^{gr} = \prod_{p \in \mathbb Z} \cC$ the $\infty$-category of graded objects in $\cC$. If $\cC$ has a symmetric monoidal structure, we will also implicitly consider $\cC^{gr}$ with its induced symmetric monoidal structure. 

\end{itemize}

\section{Preliminaries}

\subsection{Graded mixed complexes}\label{sect:gradedmixed}
Let $A \in \cdga^{\leq 0}$ be a commutative dg algebra. The category $\dgmix_A$ is the symmetric monoidal $\infty$-category of graded mixed $A$-modules. We refer to \cite[Section 1]{CPTVV} for a detailed construction of this $\infty$-category, which is defined as the $\infty$-category associated to a model category $dg_A^{gr,\epsilon}$. Its objects are graded complex $\{M(i)\}_{i \in \mathbb Z}$, together with a mixed structure $\e$, that is to say a series of maps of complexes
\[ \e : M(i) \to M(i+1) \]
of degree 1, such that $\e^2=0$. Unless otherwise specified, we will only be interested in \emph{perfect} graded mixed complexes. In other words, we will always suppose that all the $M(i)$ are perfect complexes of $A$-modules.

Consider the functor
\[ \triv \colon \dg_A \to \dgmix_A \]
sending an $A$-module $M$ to the same object with trivial graded mixed structure. In other terms, $\triv(M)$ is just $M$ concentrated in weight 0, and $\epsilon$ is identically zero. Following \cite[Section 1.1]{CPTVV} and \cite[Section 1.2]{MS}, we give the following definition.

\begin{defn}
\begin{itemize}
\item The right adjoint to the functor $\triv$ is called the \emph{realization functor}
\[ | - | \colon \dgmix_A \to \dg_A. \]
\item The left adjoint to the functor $\triv$ is called the \emph{left realization functor}
\[ | - |^l \colon \dgmix_A \to \dg_A. \]
\end{itemize}
\end{defn}

Let $M \in \dgmix_A$ be a graded mixed complex. Then one has the following explicit model for the realization functor:
\[ |M| \simeq \prod_{p \geq 0} M(p), \]
where the differential is twisted by the mixed structure of $M$. Similarly, one also have an analogous model for the left realization functor
\[ |M|^l \simeq \bigoplus_{p\leq 0} M(p), \]
where again the differential is twisted by the the mixed structure of $M$. 

The category $\dgmix_A$ admits a natural symmetric monoidal structure, defined weight-wise by 
\[ (M \otimes_A M' )(p) := \bigoplus_{i+j=p} M(i) \otimes_A M(j),\]
where the mixed structure on $M \otimes_A M'$ is the natural one. Moreover, given two objects $M, M' \in \dgmix_A$, we can consider an internal object of morphisms $\underline \Hom_{\dgmix_A}(M,M')$, whose weight components are defined by
\[ \underline \Hom_{\bdgmix_A}(M,M')(p) := \prod_{q \in \Z}\underline \Hom_A(M(q), M'(p+q)), \]
where $\underline \Hom_A(-,-)$ denotes the internal Hom object in $A$-modules. In the special case where $M' \simeq \triv(A)$ is the monoidal unit of $\dgmix_A$, we use the shorter notation
$M^{\vee} = \underline \Hom_{\dgmix_A}(M,A)$
for the dual of $M$.

Recall that we are implicitly assuming that the weight components of our graded mixed complexes are perfect. In particular, the tensor product interacts nicely with duals, and we get natural identifications
\begin{equation}\label{eq:1}
\vspace{5pt}(M^{\vee})^{\vee} \simeq M, \ \ (M \otimes_A N)^{\vee} \simeq M^{\vee} \otimes_A N^{\vee},
\end{equation}
for every $M,N \in \dgmix_A$. Moreover, an easy computation also shows that
\begin{equation}\label{eq:2}
\underline \Hom_{\dgmix_A}(M,N) \simeq \underline \Hom_{\dgmix_A}((N)^{\vee},M^{\vee}),
\end{equation}
again for all $M,N \in \dgmix_A$.

Notice however that since $M,N$ are in general unbounded in the weight direction, we cannot expect to be able to identify the internal Hom object $ \underline\Hom_{\dgmix_A}(M,N) $ with the tensor product $M^{\vee} \otimes_A N$. On the other hand, a straightforward check tells us that if we suppose $M$ to be bounded in the weight direction, then we do have
\begin{equation}\label{eq:3}
\underline\Hom_{\dgmix_A}(M,N) \simeq M^{\vee} \otimes_A N.
\end{equation}

The symmetric monoidal structure on the $\infty$-category $\dgmix_A$ allows us to consider commutative algebras inside it. In particular, we will denote by $\CAlg_{\Perf}^{gr,\epsilon}(A)$ the $\infty$-category of commutative algebras in $\dgmix_A$. A detailed construction of this $\infty$-category can be found in \cite[Sections 1.1 and 1.2]{CPTVV}. The objects of $\CAlg_{\Perf}^{gr,\epsilon}(A)$ are thus identified with graded mixed $A$-modules $(\{ B(i) \}_{i \in \mathbb Z}, \epsilon)$, with $B(i) \in \Perf(A)$, together with a series of multiplication maps
\[ m: B(i) \otimes_A B(j) \to B(i+j) \]
which are associative and commutative. Moreover, $m$ and $\epsilon$ have to be compatible, in the sense that $\epsilon$ is a derivation with respect to the product $m$.

\subsection{Moduli stacks of perfect complexes}\label{sect:perfect}



Let $\bPerf$ be the classifying stack of perfect complexes, as studied in \cite{TVa}. Recall that the values of $\bPerf$ on derived affines have the following explicit description:
\[ \Map_{\dSt}(\Spec A, \bPerf) \simeq \Perf(A)^{\sim},\]
where $\Perf(A)^{\sim}$ is the maximal groupoid contained in the $\infty$-category $\Perf(A)$ of perfect complexes on $\Spec A$.

\begin{remark}\label{rmk:perfstackcats}
One can also start from the slightly more general derived stack $\underline{\bPerf}$ with values in $\infty$-categories, whose evaluation at $\Spec A$ is the whole category $\Perf(A)$. This can be constructed using \cite[Section 1]{TVeTraces}. Then the usual (i.e. with values in $\sSet$) derived stack $\bPerf$ is obtained by composing with the maximal $\infty$-subgroupoid functor $(-)^{\sim}$.
\end{remark}
\begin{defn}
Let $X$ be a derived Artin stack. 
The moduli stack $\bPerf(X)$ of perfect complexes on $X$ is the internal mapping stack
\[ \bPerf(X) := \bMap_{\dSt}(X, \bPerf) \]
in the closed $\infty$-category $\dSt$ of derived stacks.
\end{defn}

Notice that by definition of the mapping stack, for every $\Spec A \in \dAff$ we have equivalences
\[ \Map_{\dSt}(\Spec A, \bPerf(X)) \simeq \Perf(X \times \Spec A)^{\sim}\]
of simplicial sets. In other terms, $A$-points of $\bPerf(X)$ can be identified with perfect complexes on $X \times \Spec A$.
We will therefore often use the same notation for a perfect complex on $X \times \Spec A$ and a map $\Spec A \to \bPerf(X)$.

By functoriality of the mapping stack, for every map $f : X \to Y$ of derived stacks we get an induced morphism $\boldsymbol{f}^* : \bPerf(Y) \to \bPerf(X)$ of derived stack. When evaluated at $A$-points, the morphism $\boldsymbol{f}^*$ is induced by the pullback $\infty$-functor of perfect complexes
\[ (f \times \id)^* : \Perf(Y \times \Spec A) \to \Perf(X \times \Spec A). \]

Moreover, one also has a derived stack $\bPerf^{\Delta^1}$ whose $A$-points correspond to (equivalence classes of) maps of $A$-perfect complexes. As before, $\bPerf^{\Delta^1}$ is the underlying derived stack in simplicial sets of a more general derived stack $\underline{\bPerf}^{\Delta^1}$ with values in $\infty$-categories, which sends $\Spec A$ to $\Arr(\Perf(A))$. In particular, we have an equivalence
\[ \Map_{\dSt}(\Spec A, \bPerf^{\Delta^1}) \simeq (\Arr(\Perf(A)))^\sim \]
of simplicial sets. We stress that $\bPerf^{\Delta^1}(A)$ is the space of \emph{all} morphisms between perfect complexes, and not just equivalences. Notice that $\bPerf^{\Delta^1}$ comes equipped with a couple of projections
\[ (s,t) : \bPerf^{\Delta^1} \to \bPerf \times \bPerf \]
which intuitively send a morphism to its source and its target.


\subsection{$\cO$-compact and $d$-oriented derived stacks}

We recall here the notion of $\cO$-compactness for derived stacks, following \cite[Section 2.1]{PTVV}.

\begin{defn}[see \cite{PTVV}, Definition 2.1]\label{defn:O-cpt}
Let $f: X \to Y$ be a map between derived stack. We say that $f$ is \emph{$\cO$-compact} if for every derived affine stack $Z= \Spec A$ over $Y$, the following two condition are satisfied:
\begin{enumerate}
\item $\cO_{X \times_Y Z}$ is compact in $\QCoh(X \times_Y Z)$.
\item If $p_Z: X \times_Y Z \to Z$ is the natural projection, the pushforward $(p_Z)_*$ sends perfect modules to perfect modules.
\end{enumerate}
We say that a derived stack $X$ is $\cO$-compact if its structural map $p: X \to \Spec k$ is $\cO$-compact in the sense above. 
\end{defn}

\begin{remark}
One can check that if $f: X \to Y$ is $\cO$-compact, then for any derived stack $Z$ over $Y$ (not necessarily affine) the induced pushforward $\QCoh(X \times_Y Z) \to \QCoh(Z)$ preserves perfect complexes.
\end{remark}

\begin{remark}
If $X$ is supposed to be $\cO$-compact, then in this case the global section functor $p_*$ sends perfect $\cO_X$-modules to perfect $k$-modules. More generally, if $f:X \to Y$ is $\cO$-compact, we have an induced morphism of derived stacks $\boldsymbol{f}_* : \bPerf(X) \to \bPerf(Y)$.
\end{remark}

We now recall another important notion that was introduced in \cite{PTVV}.

Fix an integer $d \in \mathbb Z$. 
Let $f: X \to Y$ be an $\cO$-compact map of derived stacks. Suppose we are given a map $\eta \colon f_*\cO_X \to \cO_Y[-d]$ of perfect $\cO_Y$-modules, and let $Z=\Spec A $ be a derived affine stack over $Y$. If we denote by $p_A$ the induced map
\[ p_A : X \times_Y Z \to Z, \]
the pullback of $\eta$ along $Z \to Y$ produces a morphism
\[ \eta_A \colon (p_A)_*\cO_{X \times_Y Z} \to A[-d] \]
of perfect $A$-modules.
Take any $\cF \in \Perf(X \times_Y Z)$. 
We get a naturally induced pairing
\[ \xymatrix{
(p_A)_*\cF \otimes_{A} (p_A)_*\cF^{\vee} \ar[r] & (p_A)_*(\cF \otimes_{\cO_{X\times_Y Z}} \cF^\vee) \ar[r] & (p_A)_*\cO_{X\times_Y Z} \ar[r]^{\ \ \ \eta_A} & A[-d]} \]
between $(p_A)_*\cF$ and $(p_A)_*(\cF^\vee)$, where the first map on the left comes from the fact that 
\[(p_A)_* : \Perf(X \times_Y Z) \to \Perf(Z)\]
is lax-monoidal (since its left adjoint $(p_A)^*$ is symmetric monoidal). This in turn induces a morphism
\[ (p_A)_*(\cF^\vee) \to ((p_A)_*\cF)^\vee[-d]. \]


\begin{defn}[see \cite{PTVV}, Definition 2.4]
Let $f : X \to Y$ be an $\cO$-compact map of derived stacks, and let $ d \in \bZ$.
\begin{itemize}
\item The space $\mathrm{PrOrient}(f,d)$ of \emph{$d$-preorientations} on $f$ is defined to be the mapping space
\[ \mathrm{PrOrient}(f,d) := \Map_{\cO_Y}(f_*\cO_X, \cO_Y[-d]) \]
in the category of $\cO_Y$-modules.
\item A point $\eta \in \mathrm{PrOrient}(f,d)$ is said to be \emph{non-degenerate} if for any derived affine stack $Z = \Spec A$ over $Y$ and any $\cF \in \Perf(X \times_Y Z)$, the induced morphism
\[ (p_A)_*\cF^{\vee} \to ((p_A)_*\cF)^\vee[-d] \]
is an equivalence of $A$-modules.
\item The space $\mathrm{Orient}(X,d)$ of \emph{$d$-orientations} on $f$ is the subspace of $\mathrm{PrOrient}(X,d)$ given by the union of connected components of non-degenerate points.
\end{itemize}
\end{defn}


If $Y = \Spec k$, we will simply use the notations $\mathrm{PrOrient}(X,d)$ and $\mathrm{Orient}(X,d)$, and call them the spaces of $d$-preorientations and of $d$-orientations on $X$.

\begin{defn}
Let $X$ be an $\cO$-compact derived stack.
\begin{itemize}
\item
We define the derived stack $\mathbf{PrOrient}(X,d)$ of $d$-preorientations on $X$
to be the following fiber product in $\dSt$:
\[ \xymatrix{
\mathbf{PrOrient}(X,d) \ar[d] \ar[rrr] & & & \bPerf^{\Delta^1} \ar[d]^{(s,t)} \\
\Spec k \ar[r]^{\cO_X \ \ \ \ \ \ \ \ \ \ \ \ \ \ \ \ \ } & \bPerf(X) \simeq \bPerf(X) \times  \Spec k  \ar[rr]^{\ \ \ \ \ \ \ \ \ \ \ \boldsymbol{f}_* \times k[-d]} & & \bPerf \times \bPerf,
}\]
where the bottom left map corresponds to the perfect complex $\cO_X \in \Perf(X)$, and the bottom right map is the product of $\boldsymbol{f}_*: \bPerf(X) \to \bPerf$ and of the $k$-point of $\bPerf$ corresponding to $k[-d]$.
\item The derived stack $\mathbf{Orient}(X,d)$ of $d$-orientations on $X$ is the derived substack of $\mathbf{PrOrient}(X,d)$ given by non-degenerate preorientations.
\end{itemize}
\end{defn}

Let $\Spec A \in \dAff$. By definition, we can identify the $A$-points of $\mathbf{PrOrient}(X,d)$ and of $\mathbf{Orient}(X,d)$ with the spaces $\mathrm{PrOrient}(p_A, d)$ and $\mathrm{Orient}(p_A,d)$ respectively, where $p_A$ is the projection $X \times \Spec A \to \Spec A$.

\begin{remark}
We could have constructed both stacks $\mathbf{PrOrient}(X,d)$ and $\mathbf{Orient}(X,d)$ more directly, starting from the observation that for every map $A \to B$ of commutative dg algebras, there are well defined morphisms
\[ \mathrm{PrOrient}(p_A,d) \to \mathrm{PrOrient}(p_B,d) \]
and
\[ \mathrm{Orient}(p_A,d) \to \mathrm{Orient}(p_B,d) \]
in the homotopy category of simplicial sets.
\end{remark}



\subsection{Quadratic and symplectic structures on derived stacks}

In this section we recall the notion of quadratic and symplectic structures on derived Artin stacks, following \cite{Ve} and \cite{PTVV}.

The following is essentially Definition 3.14 in \cite{Ve}.

\begin{defn}
Let $X$ be a derived Artin stack, and let $\bL_X \in \QCoh(X)$ be its cotangent complex. The \emph{space of $n$-shifted quadratic structures on $X$} is
\[ \mathrm{QF}(X,n) := \Map_{\QCoh(X)}(\cO_X[-n] , \Sym^2_{\cO_X} (\bL_X) ). \]
\end{defn}

\begin{remark}
The above Definition is a very mild generalization of \cite[Definition 3.14]{Ve}. The only difference here is that we allow the cotangent complex $\bL_X$ of $X$ to be possibly not perfect. In the special case of $X$ being locally finitely presented, then the two definition are clearly equivalent.
The situation here is totally analogous to the case of shifted Poisson structures (see \cite{Me} and \cite[Remark 1.4.10]{CPTVV}).
\end{remark}

Let $X$ be a derived Artin stack. Consider the graded dg-module
\[ \DR(X) := \Gamma(X,\Sym_{\cO_X}(\bL_X[-1])) \]
where the additional weight grading is given by the $\Sym$. Since as already mentioned the functor $\Gamma(X,-)$ is lax monoidal, $\DR(X)$ is naturally a commutative graded dg algebra. Then the de Rham differential turns $\DR(X)$ into a graded mixed algebras, that is to say a commutative algebra in the category $\dgmix$.

The following is essentially Definition 1.12 in \cite{PTVV} (see also \cite[Definition 2.4.14]{CPTVV}).

\begin{defn}
The space of \emph{closed $p$-forms of degree $n$ on $X$} is the mapping space
\[ \mathcal A^{p,cl}(X,n) := \Map_{\dgmix}(k(p)[-n-p], \DR(X)), \]
where $k(p)[-n-p]$ is the trivial graded mixed module $k$ sitting in weight $p$ and cohomological degree $n+p$.
\end{defn}

Suppose moreover that $X$ is locally of finite presentation, so that its cotangent complex $\bL_X$ is perfect. In particular, consider its dual $\bT_X= \bL_X^{\vee}$ in $\QCoh(X)$. We say that a closed 2-form of degree $n$ is an $n$-symplectic structure if the induced map $\bT_X \to \bL_X[-n]$ is an equivalence.

\section{The moduli stack of Lie coalgebroids}

The goal of this section is to construct an appropriate moduli stack of Lie coalgebroids on a given derived stack $X$. 

\begin{defn}
Let $X \in \dSt$. The $\infty$-category $\CAlg_{\Perf}^{gr}(X)$ is the $\infty$-category of commutative algebras in the symmetric monoidal $\infty$-category $\Perf(X)^{gr}$. If $X \simeq \Spec A$ is a derived affine, we will simply use the notation $\CAlg_{\Perf}^{gr}(A)$.
\end{defn}

Notice that for every map $A  \to B$ of commutative dg algebras, the induced pullback $\infty$-functor $\Perf(A) \to \Perf(B)$ is symmetric monoidal, and thus we get an induced base change $\infty$-functor
\[ \CAlg_{\Perf}^{gr}(A) \to  \CAlg_{\Perf}^{gr}(B), \]
which is an equivalence whenever $A\to B$ is an equivalence itself.
It follows that, as discussed in Remark \ref{rmk:perfstackcats}, we can apply the methods of \cite[Section 1]{TVeTraces} to define a cofibered $\infty$-category over $\dAff^{op}$, or equivalently a prestack $\underline{\bCAlg}^{gr}_{\Perf}$ with values in $\infty$-categories. This functor sends a commutative dg algebra $A$ to the associated $\infty$-category $\CAlg^{gr}_{\Perf}(A)$.
The prestack $\underline{\bCAlg}^{gr}_{\Perf}$ is a derived stack with respect to the (derived) \'etale topology, simply because $\underline{\bPerf}^{gr}$ was already a derived stack in $\infty$-categories.
Composing with the underlying maximal $\infty$-subgroupoid functor $(-)^\sim$,
we get an induced classifying stack of graded perfect commutative algebras $\bCAlg^{gr}_{\Perf} \in \dSt$, whose space of $A$-points is equivalent to $\CAlg_{\Perf}^{gr}(A)^{\sim}$.

\begin{defn}
If $X$ is a derived stack, then the \emph{classifying stack $\bCAlg_{\Perf}^{gr}(X)$ of graded perfect commutative algebras on $X$} is the mapping stack
\[ \bCAlg_{\Perf}^{gr}(X) := \bMap_{\dSt}(X, \bCAlg_{\Perf}^{gr}). \]
\end{defn}

\begin{remark}
The functoriality of the mapping stack once again implies that the construction $X \mapsto \bCAlg_{\Perf}^{gr}$ is functorial. More specifically, if $f: X\to Y$ is a map of derived stack, then there exists a pullback morphism
\[ f^* : \bCAlg_{\Perf}^{gr}(Y) \to \bCAlg_{\Perf}^{gr}(X) \]
of derived stacks. If moreover we suppose $f$ to be $\cO$-compact, we also get a pushforward morphism
\[ f_* : \bCAlg_{\Perf}^{gr}(X) \to \bCAlg_{\Perf}^{gr}(Y). \]
Notice that the existence of $f_*$ uses the fact that the pushforward of perfect complexes is a lax-monoidal functor, and as such it preserves commutative algebra objects.
\end{remark}

By definition, we have an equivalence
\[ \Map_{\dSt}(\Spec A, \bCAlg_{\Perf}^{gr}) \simeq \CAlg_{\Perf}^{gr}(X \times \Spec A) \]
of simplicial sets.
Notice that there is a natural map of derived stacks
\[ \bPerf(X) \to \bCAlg_{\Perf}^{gr}(X), \]
which is induced by the $\infty$-functors
\[ \xymatrix{
\Perf(X \times \Spec A) \ar[rr]& & \Perf(X \times \Spec A)^{gr} \ar[rrr]^{\Sym_{\cO_{X \times \Spec A}}(-)} & & & \CAlg^{gr}_{\Perf}(X \times \Spec A)
} \]
where the map on the left sends a perfect complex $\mathcal F$ to the complex $\mathcal F[-1]$ sitting on weight $1$, and the second functor is the free commutative algebra functor.
In particular, if we suppose that $X$ is moreover $\cO$-compact in the sense of Definition \ref{defn:O-cpt}, we can consider the composition of maps of derived stacks
\[ \bPerf(X) \to \bCAlg^{gr}_{\Perf}(X) \to \bCAlg^{gr}_{\Perf}, \]
where the map on the right is now the pushforward induced by $X \to \Spec k$, which exists because $X$ is $\cO$-compact.
The composition above corresponds, for every $\Spec A \in \dAff$, to sending a perfect complex $\cF \in \Perf(X \times \Spec A)$ to the perfect graded commutative algebra $(p_A)_*\Sym_{\cO_X}(\cF[-1]) \in \CAlg^{gr}_{\Perf}(A)$, where $p_A : X \times \Spec A \to \Spec A$ is the projection on the second term. Note that the weight grading on $(p_A)_*\Sym_{\cO_X}(\cF[-1])$ coincides with the natural one induced by the symmetric powers.

Moreover, let $\bdgmix_{\Perf}$ and $\bCAlgmix_{\Perf}$ be the classifying stack of graded mixed complexes and of graded mixed commutative algebras respectively, constructed in the same way as $\bPerf$ and $\bCAlg^{gr}_{\Perf}$. More specifically, we have equivalences
\[ \Map_{\dSt}(\Spec A, \bdgmix_{\Perf}) \simeq (\dgmix_A)^\sim, \ \ \ \ \ \ \Map_{\dSt}(\Spec A, \bCAlgmix_{\Perf}) \simeq (\CAlg^{gr,\epsilon}(A))^\sim.\] 
Then we have a natural forgetful map
\[ \bCAlgmix_{\Perf} \to \bCAlg_{\Perf}^{gr} \]
which simply forgets the mixed structure.


Using these stacks, we can now give the following definition.

\begin{defn}
Let $X$ be a derived $\cO$-compact stack. The \emph{moduli stack of perfect Lie coalgebroids on $X$} is the fiber product
\[ \xymatrix{
\bCoAlgd(X) \ar[r] \ar[d] & \bPerf(X) \ar[d] \\
\bCAlgmix_{\Perf} \ar[r] & \bCAlg^{gr}_{\Perf}
}\]
in the category $\dSt$ of derived stacks.
\end{defn}

\begin{remark}
Even though our notation $\bCoAlgd(X)$ doesn't suggest it, an important point is that we are only working with Lie coalgebroids whose underlying module is perfect. One could of course avoid this restriction, and give a similar definition for Lie coalgebroids on a general (i.e. not necessarily $\cO$-compact) derived stack $X$.
\end{remark}

By definition, a $k$-point of $\bCoAlgd(X)$ corresponds to a perfect complex $\cF$ on $X$ with a mixed structure on $\Gamma(X, \Sym_{\cO_X}(\cF[-1]))$. These are precisely perfect Lie co-algebroids on $X$, that is to say perfect complexes on $X$ whose duals are Lie algebroids. The mixed structure here is the data corresponding to the Chevalley-Eilenberg differential on the CE algebra for the Lie algebroid. 

In general, $A$-points of $\bCoAlgd(X)$ are given by a perfect complex $\cF$ on $X \times \Spec A$, together with a $A$-linear mixed structure on the graded algebra
\[ (p_A)_*(\Sym_{\cO_{X \times \Spec A}}(\cF[-1])),  \]
where $p_A : X \times \Spec A \to \Spec A$ is the natural projection. 

\begin{remark}
Suppose that $X$ is a derived Artin stack. It follows that there is a distinguished point of $\bCoAlgd(X)$, i.e. a canonical map
\[ \Spec k \to \bCoAlgd(X) \]
representing the Lie co-algebroid $\bL_X$ (that is to say, the dual of the tangent Lie algebroid $\bT_X$).
\end{remark}

We denote by $(\bdgmix_{\Perf})^{\Delta^1}$ the derived stack of morphisms of $\bdgmix_{\Perf}$, constructed similarly to $\bPerf^{\Delta^1}$ in Section \ref{sect:perfect}. More specifically, the space of $A$-points of $(\bdgmix_{\Perf})^{\Delta^1}$ is equivalent to $(\Arr(\dgmix(A))^{\sim}$,
where $\Arr(\dgmix_A)$ is the $\infty$-category of morphisms in $\dgmix_A$. By definition, the stack $(\bdgmix_{\Perf})^{\Delta^1}$ comes equipped with two natural maps $s$ and $t$ (for ``source'' and ``target'')
\[ \xymatrix{
\bdgmix_{\Perf} & \ar[l]_{s\ \ } (\bdgmix_{\Perf})^{\Delta^1} \ar[r]^{\ \ t} & \bdgmix_{\Perf}
}\]
which remember only the source or the target of the points of $(\bdgmix_{\Perf})^{\Delta^1}$.

Now consider the induced composition of maps of stacks
\[ \bCoAlgd(X) \to \bCAlgmix_{\Perf} \to \bdgmix_{\Perf} \]
where the first map is the one coming from the definition of $\bCoAlgd(X)$, and the second forgets the algebra structure, and just retains the underlying graded mixed module. Let us call $\phi$ this composition. 

\begin{defn}
The moduli stack $\mathcal Y_n$ of \emph{$n$-pre-symplectic Lie co-algebroids} on $X$ is the fiber product
\[ \xymatrix{
\mathcal Y_n \ar[r] \ar[d] & \bCoAlgd(X) \ar[d]^{(k(2)[-n-2], \phi)} \\
(\bdgmix_{\Perf})^{\Delta^1} \ar[r]^{(s,t)\ \ \ \  } & \bdgmix_{\Perf} \times \bdgmix_{\Perf} 
} \]
in the category $\dSt$ of derived stacks.
\end{defn}

Again by definition, a $k$-point of $\mathcal Y_n$ is given by the following data
\begin{itemize}
\item a perfect module $\cF$ on $X$
\item a mixed structure on the graded commutative algebra $\Gamma(X, \Sym_{\cO_X}(\cF[-1]))$
\item a map of graded mixed modules
\[ k[-n-2](2) \longrightarrow \Gamma(X, \Sym_{\cO_X}(\cF[-1])). \]
\end{itemize}
In general, $A$-points of $\mathcal Y_n$ are perfect modules $\cF$ on $X \times \Spec A$, such that
\[ (p_A)_*\Sym_{\cO_{X \times \Spec A}}(\cF[-1]) \]
is a $A$-linear graded mixed commutative algebra, together with a map of graded mixed complexes
\[ A(2)[-n-2] \longrightarrow (p_A)_*\Sym_{\cO_{X \times \Spec A}}(\cF[-1]). \]

\begin{defn}
Finally, suppose $X$ is a derived Artin stack. The moduli stack of $n$-pre-symplectic structures $\PrSymp(X,n)$ is the fiber product
\[ \xymatrix{
\PrSymp(X,n) \ar[r] \ar[d] & \mathcal Y_n \ar[d] \\
\Spec k \ar[r] & \bCoAlgd(X) 
} \]
where the bottom map is the one representing $\bL_X$.
\end{defn}

\section{The cotangent complex of $\Symp(X,n)$}

In this section, we study the geometry of the stack $\PrSymp(X,n)$ in more detail: in particular, we compute its cotangent complex, following the explicit definition of \cite[Section 1.2.1]{TVe}.

Throughout this section, $X$ will be an $\cO$-compact derived stack.
As a first remark, we notice that by definition this stack fits in a cartesian square
\[
\xymatrix{
\PrSymp(X,n) \ar[r] \ar[d] & \Spec k \ar[d]^{k(2)[-n-2] \times \DR(X)} \\
(\bdgmix_{\Perf})^{\Delta^1} \ar[r]^{(s,t) \ \ } & \bdgmix_{\Perf} \times \bdgmix_{\Perf}.
}
\]
In other words, a $k$-point of $\PrSymp(X,n)$ is just a map
\[ k(2)[-n-2] \to \DR(X) \]
of graded mixed complexes, i.e. a closed 2-form of degree $n$ on $X$. More generally, an $A$-point of $\PrSymp(X,n)$ is a degree $n$ closed 2-form of $X \times \Spec A$ relative to $A$, or equivalently a map of graded mixed $A$-modules
\[A(2)[-n-2] \to p_*(\Sym_{\cO_{X \times \Spec A}} \bL_{X\times \Spec A / \Spec A}[-1]) \simeq \DR(X) \otimes_k A. \]

\begin{remark}
Even though $k$-points of $\PrSymp(X,n)$ are the same as $k$-points of $\mathcal A^{2,cl}(X,n)$, the two stacks are not equivalent, as $A$-points of $\mathcal A^{2,cl}(X,n)$ are just degree $n$ closed 2-forms on $X \times \Spec A$ relative to $k$. 
\end{remark}

Let us now consider an $A$-point $\omega$ of $\PrSymp(X,n)$, corresponding to a map
\[\omega \colon A(2)[-n-2] \to \DR(X) \otimes_k A. \]
With a slight abuse of notation, let us also denote by $\omega$ the induced $A$-point of $(\bdgmix_{\Perf})^{\Delta^1}$. Using the fact the the above diagram is cartesian, we know that if the bottom map $(s,t)$ has a relative cotangent complex at $\omega$, then also $\bL_{\PrSymp(X,n), \omega}$ exists, and moreover we have an equivalence
\[ \bL_{\PrSymp(X,n), \omega} \simeq \bL_{(s,t), \omega}. \]

\begin{prop}
Let again $(s,t)$ be the map of derived stacks
\[ (\bdgmix_{\Perf})^{\Delta^1} \longrightarrow \bdgmix_{\Perf} \times \bdgmix_{\Perf} \]
sending a morphism to its source and target. Let $f : \Spec A \to (\bdgmix_{\Perf})^{\Delta^1}$  correspond to a map $f : E \to F$ in $\bdgmix_{\Perf}(A)$. Then $(s,t)$ admits a cotangent complex at the point $f$, which is given by
\[ \bL_{(s,t), f} \simeq | \underline{\Hom}_{\dgmix_A}(E,F)^{\vee}|^l ,\]
where $|-|^l$ is the left realization of Section \ref{sect:gradedmixed}.
\end{prop}

\begin{proof}
Let $M$ be an $A$-module. A straightforward computation shows that the space of relative derivations can be expressed as
\[ \Der_{(s,t)}(A,M) \simeq \Map_{\dgmix_A}(E , F \otimes_A M), \]
where $M$ is taken with the trivial graded mixed structure. Moreover, in view of the identifications (\ref{eq:1}) and (\ref{eq:2}) in the category $\dgmix_A$, we get
\[ \Map_{\dgmix_A}(E, F \otimes_k M) \simeq \Map_{\dgmix_A}(\underline \Hom_{\dgmix_A}(E,F)^\vee, M). \]
By definition of left realization, we conclude by adjunction that
\[ \Der_{(s,t)}(A,M) \simeq \Map_{\dg_A}(| \underline{\Hom}_{\dgmix_A}(E,F)^\vee|^l, M), \]
which proves the proposition.
\end{proof}

An immediate consequence of the above result is the following corollary.

\begin{prop}\label{prop:localcotangent}
The derived stack $\PrSymp(X,n)$ admits a cotangent complex in every point. In particular, given an $A$-point $\omega$ of $\PrSymp(X,n)$, we have
\[ \bL_{\PrSymp(X,n), \omega} \simeq \bigoplus_{p \geq 2} \underline \Hom_A (A[-n-2], \DR(X)(p) \otimes_k A)^\vee. \]
\end{prop}
\begin{proof}
Using that $X$ is $\cO$-compact, this is a simple combination of the above Proposition and the observation that we have an equivalence
\[ | \underline \Hom_{\dgmix_A}(A(2)[-n-2], \DR(X )\otimes A )^\vee|^l \simeq \bigoplus_{p \geq 2} \underline \Hom_A (A[-n-2], \DR(X)(p) \otimes_k A)^\vee, \]
where the right hand side is endowed with the twisted differential coming from the mixed structure on $\DR(X) \otimes A$.
\end{proof}

\begin{remark}
Suppose $X$ is such that $\DR(X)$ is a bounded graded mixed perfect complex. Then the sum appearing in the statement of the above Proposition is in fact finite, and thus $\PrSymp(X,n)$ has a perfect cotangent complex at every point. In particular, it is now easy to see that one has
\[ \bT_{\PrSymp(X,n), \omega} \simeq \underline \Hom_{A}(A[-n-2], |\DR(X)^{\geq 2} \otimes_k A |). \]
Notice that this is in line with the content of the conjecture of \cite[Remark 3.15]{Ve}, as the right hand side is precisely the complex of closed 2-forms of degree $n$ on $X \times \Spec A$, relative to $\Spec A$. 
\end{remark}

\begin{prop}
The derived stack $\PrSymp(X,n)$ has a global cotangent complex.
\end{prop}
\begin{proof}
Indeed, suppose we have maps
\[ 
\xymatrix{
\Spec B \ar[r]^{\phi} & \Spec A \ar[r]^{f\ \ \ \ \ \  } & \PrSymp(A,n).
} \]
In view of \cite[Definition 1.4.1.7]{TVe}, we need to check that the induced map
\[\phi^*\bL_{\PrSymp(X,n), f} \to \bL_{\PrSymp(X,n), f \circ \phi} \]
is an equivalence of $B$-modules.
But by the above proposition, this morphism comes from a map of graded mixed complexes
\[ \underline{\Hom}_{\dgmix_A}(E,F)^\vee \otimes_A B \to \underline{\Hom}_{\dgmix_B}(E \otimes_A B , F \otimes_A B)^\vee,\]
where in this particular case $E = A(2)[-n-2]$ and $F = \DR(X) \otimes A.$
Since $E$ is bounded, we know from identification (\ref{eq:3}) that
\[ \underline \Hom_{\dgmix_A}(E,F) \simeq E^\vee \otimes_A F, \]
and thus the above map is an equivalence, concluding the proof.
\end{proof}

\section{The quadratic form}

In this section we state and prove our main result. Namely, we show that one can endow $\Symp(X,n)$ with a canonical quadratic form, extending the one of \cite{FH}. We keep assuming that $X$ is $\cO$-compact.

Let $X$ be a derived Artin stack, locally of finite presentation. Let $Y= \Spec A$ be a derived affine scheme, and let $\omega: Y \to \PrSymp(X,n)$ be represented by an $n$-shifted pre-symplectic structure on $X \times Y$, relative to $Y$. Then we say that $\omega$ is \emph{non-degenerate} if the induced map of $\cO_{X \times Y}$-modules
\[ \omega^{\sharp} \colon \bT_{X \times Y/Y}[-n] \to \bL_{X\times Y / Y} \]
is an equivalence.

\begin{defn}
Let $X$ be a derived Artin stack, locally of finite presentation. The derived stack $\Symp(X,n)$ of $n$-shifted symplectic structures on $X$ is the substack of $\PrSymp(X,n)$ composed of non-degenerate pre-symplectic structure. 
\end{defn}

Notice that $\Symp(X,n)$ is exactly the moduli stack involved in the conjecture in \cite[Remark 3.15]{Ve}.

\



The quadratic structure constructed in \cite{FH} uses in a crucial way the fact that classically symplectic structures induce orientations on the manifold. 
In order to reproduce the arguments of \cite{FH} in our context, we will need a replacement for this result.

Let $p : X \to \Spec k$ be the canonical projection. Let $d\in \mathbb Z$, and denote by $F^d_X$ the prestack given by the following fiber product
\[ \xymatrix{ F^d_X \ar[rrr] \ar[d]  & & & \bPerf^{\Delta^1} \ar[d]^{(s,t)}  \\
\Spec k \ar[r]^{\cO_X \ \ \ \ \ \ \ \ \ \ \ \ \ \ \ \ } & \bPerf(X) \simeq \Spec k \times \bPerf(X) \ar[rr]^{\ \ \ \ \ \ \ \ \ \ \ k[-d] \times p_*}& &\bPerf \times \bPerf.
} \]

By definition, we have an equivalence
\[ \Map_{\dSt}(\Spec A, F^d_X) \simeq \Map_A(A, (p_A)_*\cO_{X \times \Spec A}[d]), \]
where $p_A : X \times \Spec A \to \Spec A$ is as usual the standard projection. One can think of $F^d_X$ as a stack of ``degree $d$ global functions'' on $X$. Notice that since $(p_A)_*\cO_{X \times \Spec A}$ is an $A$-algebra for every $A$, any point $\phi \in F^d_X(A)$ induces a multiplication map
$m_{\phi} : F^0_X(A) \to F^d_X(A)$.
\begin{defn}\label{def:finalhyp}
Suppose $X$ is a $\cO$-compact derived stack, equipped with a map of derived stacks
\[ \eta \colon  \Symp(X,n)  \to  \mathbf{Orient}(X,d)\]
sending a point $\omega \in \Symp(X,n)(A)$ to a $d$-orientation $\eta_\omega \in \mathrm{Orient}(p_A,d)$, where $p_A$ is the projection
\[ X \times \Spec A \to \Spec A. \]
Moreover, suppose that we have a map
\[ v:  \Symp(X,n) \to F^d_X, \]
which we suppose to be non-zero.
Then for every $\omega \in \Symp(X,n)(A)$ and $f \in F^0_X(A)$, we can define the \emph{integral} $\int_X f$ by applying the composition
\[ \xymatrix{
F^0_X(A) \ar[r]^{m_{v(\omega)}} & F^d_X(A) \ar[r]^{\ \ \eta_\omega} & A
} \]
to the point of $F^0_X(A)$ corresponding to $f$. Notice in particular that by definition we have $\int_X(f) \in A$.
\end{defn}

\begin{thm}\label{thm:quadform}
	 Suppose that $X$ satisfies the hypothesis of Definition \ref{def:finalhyp}. Then there exists a canonical non-trivial quadratic form on $\Symp(X,n)$, extending the one of \cite{FH}.
\end{thm}

\begin{proof}
Let again $Y=\Spec A$ be a derived affine scheme. Since $\Symp(X,n)$ is an open substack of $\PrSymp(X,n)$, it follows that if $\omega$ is an $A$-point of $\Symp(X,n)$, then one has an equivalence
\[ \bL_{\Symp(X,n), \omega} \simeq \bL_{\PrSymp(X,n),\omega} \]
of $A$-modules. In particular, thanks to Proposition \ref{prop:localcotangent} we have
\[ \bL_{\Symp(X,n), \omega} \simeq | \underline \Hom_{\dgmix_A}(A(2)[-n-2], \DR(X) \otimes A)^{\vee}|^l. \]
As a consequence, by considering only the weight 0 component of the right hand side we get a morphism
\[ \underline \Hom_{A}(A[-n-2], \DR(X)(2)\otimes A)^{\vee} \longrightarrow \bL_{\Symp(X,n), \omega} \]
of $A$-modules. Remark that the source is in fact by definition the $A$-linear dual of the complex of 2-forms on $X \times \Spec A$, relative to $\Spec A$. Since $X$ is supposed to be locally of finite presentation, its cotangent complex is in particular perfect. It follows that by adjunction, we get a morphism
\[ \underline \Hom_{A}(A[-n-2], \DR(X)(2)\otimes A) \longrightarrow \Hom_{\QCoh(X\times Y)}(\bT_{X \times Y / Y}[-n] , \bL_{X \times Y / Y}) \]
of $A$-modules, where the Hom on the right hand side is the $\dg_A$-enriched Hom of $\QCoh(X\times Y)$. On the other hand, the $A$-point $\omega$ corresponds to a symplectic structure, and therefore it gives an identification $ \bT_{X \times Y/Y}[-n] \simeq \bL_{X \times Y/ Y}$ given by $\omega^{\sharp}$. Hence, we get a map (and in fact an equivalence)
\[ \Hom_{\QCoh(X\times Y)}(\bT_{X \times Y / Y}[-n] , \bL_{X \times Y / Y}) \longrightarrow (p_A)_*\underline{\End} (\bL_{X \times Y / Y }) \]
of $A$-modules, where the $\underline\End$ on the right is the internal endomorphism object of $\QCoh(X\times Y)$, and $p_A$ is as usual the projection $X \times Y \to Y$. Dualizing these last morphisms and putting all together, we end up with
\[ ((p_A)_*\underline\End( \bL_{X \times Y / Y }))^{\vee} \longrightarrow \bL_{\Symp(X,n), \omega} \]
of $A$-modules. Finally, notice that there are canonical $A$-linear maps
\[ \Sym_A^2 ((p_A)_*\underline\End(\bL_{X \times Y / Y})) \to (p_A)_*\Sym_{\cO_{X \times Y}}(\underline\End(\bL_{X \times Y / Y})) \to (p_A)_*\cO_{X \times Y} \]
where the last arrow is simply induced by the standard formula
\[ (M,N) \mapsto \frac 12 \mathrm{Tr}(MN), \]
as in \cite[Section 2]{FH}.

Applying the integration of Definition \ref{def:finalhyp}, we get a well defined
\[ \int_X \frac 12 \mathrm{Tr}(MN)  \ \in \,A,\]
which together with the above discussion produces a morphism
\[ \Sym_A^2 ((p_A)_*\underline\End(\bL_{X \times Y / Y})) \to A. \]
Dualizing this map, we eventually get a well defined quadratic structure on $\Symp(X,n)$, which clearly extends the one constructed in \cite{FH}.

\end{proof}

\begin{remark} Theorem \ref{thm:quadform} can be mildly generalized to \emph{spaces}, getting a statement which is somehow more in line with the general philosophy of derived algebraic geometry. More specifically, the same arguments in the proof of Theorem \ref{thm:quadform} show that there exists a map
\[ \Map_{\dSt}(\Symp(X,n) , \mathbf{Orient}(X,d) \times F^d_X) \to \QF(\Symp(X,n),0) \]
of simplicial sets, which sends a couple $(\eta, v)$ to the quadratic form constructed above. 
\end{remark}

\end{document}